\documentclass[12pt,a4paper]{amsart}

\usepackage{amsmath,amssymb,amsfonts,amsthm,enumerate}

\usepackage{tikz}
\tikzstyle{fekete}=[circle, draw, thin,fill=black, scale=0.4]
\tikzstyle{feher}=[circle, draw, thin,fill=white, scale=0.4]

\usepackage[longnamesfirst]{natbib}
\bibpunct[ ]{(}{)}{,}{a}{}{,}

\newtheorem{df}{Definition}
\newtheorem{lemma}[df]{Lemma}

\theoremstyle{plain}\newtheorem{theorem}[df]{Theorem}
\theoremstyle{remark}\newtheorem{remark}[df]{Remark}
\theoremstyle{remark}
\theoremstyle{plain}\newtheorem{corollary}[df]{Corollary}
\theoremstyle{plain}\newtheorem{prop}[df]{Proposition}

\title[Spectral measures of factor of i.i.d.\ processes]{Spectral measures of factor of i.i.d.\ processes on vertex-transitive graphs}
\author{\'Agnes Backhausz}
\address{\'Agnes Backhausz. {\rm MTA Alfr\'ed R\'enyi Institute of Mathematics and
E\"otv\"os Lor\'and University; Budapest, Hungary}.}  
 \email{backhausz.agnes@renyi.mta.hu}
 \author{B\'alint Vir\'ag}
 \address{B\'alint Vir\'ag. {\rm Department of Mathematics, University of Toronto, Canada and MTA Alfr\'ed R\'enyi  Institute of Mathematics, Budapest, Hungary.}}
 \email{balint@math.toronto.edu}


\date{2 September 2016}

\begin{document}

\begin{abstract}
We prove that a measure  on $[-d, d]$  is the spectral measure of a factor of i.i.d.\ process on a vertex-transitive infinite graph if and only if it is absolutely continuous with respect to the spectral measure of the graph. Moreover, we show that the set of spectral measures of factor of i.i.d.\ processes and that of $\bar d_2$-limits of factor of i.i.d.\ processes are the same. 
\end{abstract}

\maketitle
\thispagestyle{empty}


\section{Introduction}

We consider invariant random processes on vertex transitive graphs that can be performed with ''randomized local'' algorithms (factor of i.i.d.\ processes).   The aim of the paper is characterizing the covariance structures of factor of i.i.d.\ processes, in terms of the absolute continuity of their spectral measure with respect to the spectral measure of the graph.  

\subsection{Factor of i.i.d. processes}Let  $G=(V(G), E(G))$ be a vertex-transitive graph with countable vertex set $V(G)$.  We assign a random variable $X_v$ to each $v\in V(G)$. We get an {\it invariant random process} if the joint distribution is invariant under the automorphisms of $G$. These are the analogues of stationary processes on $\mathbb Z$, and they may be interesting on their own right \citep{arnaudregi, arnaud}. 

Among invariant random processes, we deal with {\it factor of i.i.d.\ processes}, see   e.g.\ \citet{gabor, HLSz} or the paper of \cite{lyons} and the references therein. Results of randomized local algorithms (or constant-time parallelized algorithms) belong to this class. Furthermore, factor of i.i.d.\ processes can be useful for 
finding large independent sets (Cs\'oka, Gerencs\'er, Harangi and Vir\'ag 2014, Harangi and Vir\'ag 2013, Hoppen and Wormald 2013), matchings on nonamenable graphs \citep{nazarov, csokalipp}, colorings, other structures (Gaboriau and Lyons 2009, Kun 2013, Backhausz and Szegedy 2014). This family of processes may also be interesting from an ergodic theoretic point of view \citep{lewis}, as they are the factors of the Bernoulli shift. To define factor of i.i.d.\ processes, loosely speaking, we start with independent and identically distributed labels (from $\mathbb R$) on the vertices. Then each vertex gets a new label, depending on the labelled rooted graph as it is seen from that vertex. The rule is fixed, and it is the same for all vertices. See Subsection \ref{fiid} for the precise definition. 

\subsection{Spectral measures, covariance structures and $\bar d_2$ limit.} We assign a finite measure on $\mathbb R$ to invariant random processes on   $G$, which is the spectral measure of the process. (Throughout the paper, we always consider Borel measures.) This measure is the spectral measure of the graphing associated to the process (Section \ref{plancherel}). This is also related to the covariance structure of the process, as the following is satisfied for a process $X$: 
\[\mathbb E( [A^k X]_o, X_o)=\langle A^k \delta_o, c_X\rangle_G=\int t^k d\mu_X(t) \qquad (k\geq 0),  \]
where $A$ is the adjacency operator, $c_X: V\rightarrow \mathbb R$ is the covariance structure assigning $\mathrm{cov} (X_o, X_v)$ to each vertex $v$, and $\mu_X$ is the spectral measure of $X$.

\medskip

 Our goal is  characterizing the set of spectral measures of factor of i.i.d.\ processes. As we will see, the same characterization holds for $\bar d_2$-limits of factor of i.i.d.\ processes. The following metric is based on Ornstein's $\bar d$-metric \citep[see e.g.][]{lyons}. 
\begin{df}\label{def:dbar}
The $\bar d_2$-distance of the random invariant processes $X, X'$ (with marginals having finite second moments) is defined as follows. 
\begin{align*}\bar d_2(X, X')^2=\min\{ \mathbb E\big[(Y_o-Y'_o)^2\big]: &\, Y\stackrel{d}= X, Y'\stackrel{d}=X',\\&\quad (Y, Y') \textrm{ is invariant}\}. \end{align*} 
\end{df}

\medskip

If the graph is the $d$-regular tree (throughout the paper, we denote by $T_d$ the  infinite $d$-regular tree), the spectral measure equals to the  Fourier transform of the covariance structure, see e.g.\ \cite{cartier, arnaudregi, arnaud, fpic}. On the other hand, in \cite{decay} a uniform exponential bound was proved for the decay of the correlation sequence of a factor of i.i.d.\ process, and the pointwise closure of the possible correlation sequences was described in the case $G=T_d$ ($d\geq 3$).

\subsection{Main results}

Our main result is the following theorem, which characterizes the spectral measures of factor of i.i.d.\ processes,  linear factor of i.i.d.\ processes (see Definition \ref{def:linfiid} and \ref{def:sphlin}) and  $\bar d_2$-limits of factor of i.i.d.\ processes (Definition \ref{def:dbar}).  
\begin{theorem} \label{tfiid} Fix an infinite vertex-transitive graph $G$. Suppose that all processes below have marginals with mean 0 and finite second moment. The following are equivalent for a finite (Borel) measure on $\mathbb R$.

$(i)$ It is absolutely continuous with respect to 
the spectral measure $\nu$ of the graph $G$. 

$(ii)$  It is the spectral measure of some linear factor of i.i.d.\ process. 

$(iii)$ It is the spectral measure of some factor of i.i.d.\ process.

$(iv)$ It is the spectral measure  of some limit of factor of i.i.d.\ process with respect to the $\bar d_2$-distance.

$(v)$ It is the spectral measure of some invariant process which is the $\bar d_2$-limit of processes with spectral measures  satisfying $(i)$. \end{theorem} 

The equivalence of $(iii)$ and $(iv)$ shows that the conjecture about $\bar d_2$-limits of factor of i.i.d.\ processes can not be refuted based on spectral measures or covariance structures.

\begin{remark}
The proof shows that the family of processes with spectral measures absolutely continuous with respect to any fixed measure on $[-1,1]$ is closed under $\bar d_2$-convergence.
\end{remark}

\medskip

Based on this theorem, we give the characterization of spectral measures of processes that are limits of factor of i.i.d.\ processes in distribution. 

\begin{theorem} \label{tlimfiid} Fix an infinite vertex-transitive graph $G$. Suppose that all processes below have marginals with mean 0 and finite second moment. The following are equivalent for a finite (Borel) measure $\mu$ on $\mathbb R$.

$(i)$ Its support is contained in the support of the spectral measure of $G$, that is, $\mathrm{supp}(\mu)\subseteq \mathrm{supp}(\nu)$.

$(ii)$  It is the spectral measure of the weak limit of some linear factor of i.i.d.\ processes. 

$(iii)$ It is the spectral measure of the weak limit of some factor of i.i.d.\ process.

\end{theorem}

As for the $d$-regular tree, Theorem 5.1 of Backhausz, Szegedy and Vir\'ag (2014) gives a description of the pointwise closure of the correlation sequences of factor of i.i.d.\ processes, which shows the equivalence of $(i)$ and $(iii)$ in Theorem \ref{tlimfiid} after some reformulation.

\begin{remark}
The family of processes that can be modelled on random $d$-regular graphs in an appropriate sense (see e.g.\ \cite{balazs}) is strictly wider than the limit of factor of i.i.d.\ processes; for example,  for large $d$, the independence ratio is twice as large as the proportion of an independent set that can be constructed with factor of i.i.d.\ \citep{gamarnik, mustazee}. However, the possible covariance structures are the same for the two families of invariant random processes, and  the results remain valid. 
\end{remark}

\subsection{Applications} 

\subsubsection*{Process spectrum} The aim of this section is describing the set of points that can be included in the spectrum of an invariant random process.

\begin{df}[Process spectrum of a graph] Let $G$ be a $d$-regular vertex transitive graph. Its process spectrum is defined as follows: 
\[\mathrm{psp}(G)=\overline{\bigcup_X \mathrm{supp}(\mu_X)}\subseteq [-d, d],\]
where the union is for all invariant random processes on $G$ with marginals having finite variance.   The process spectral radius $\varrho^+_p(G)$ is defined by $\sup\{x: x\in \mathrm {psp}(G),\ x<d\}$. 
\end{df}

If $d$ is fixed, one can ask whether there are $d$-regular graphs with process spectral radius arbitrarily close to $2\sqrt{d-1}$. 

\begin{theorem}\label{tpsp} Let $\mu$ be a finite measure with $\mathrm{supp} (\mu)\subseteq \mathrm{psp}(G)$. Then there exists an invariant random process $X$ with spectral measure $\mu$. 
\end{theorem}

For the $d$-regular tree it is known that the process spectrum is $[-d,d]$ \citep{arnaudregi, arnaud, lasser}, while the spectrum of the tree is $[-2\sqrt{d-1}, 2\sqrt{d-1}]$. In particular, Gaussian wave functions (see the definition below) exist for all $\lambda\in [-d, d]$ (Cs\'oka, Gerencs\'er, Harangi and Vir\'ag 2014, Fig\`a-Talamanca and Nebbia 1991). 

We will also prove the following statement about the connection of process spectrum and Kazhdan's property $(T)$ (see e.g.\ Bekka, de la Harpe and Valette 2008). This also shows that the behavior of process spectral radius and the behavior of spectral radius with respect to tensor product of graphs are different (take the product of two graphs, one with process spectral radius equal to $d$, and the other one less than $d$).

\begin{prop}\label{kazhdan} Let $H$ be a finitely generated infinite group, and $G$ its Cayley graph with some set of generators. Then the following are equivalent. 
\begin{enumerate}[(i)]
\item $H$ has Kazhdan's property $(T)$. 
\item The process spectral radius $\varrho_p^+(G)$ is less than $d$.
\end{enumerate}
\end{prop}

\subsubsection*{$\bar d_2$-distance and total variation distance of the spectral measures} During the proof of the main theorem, we will show the following inequality and orthogonality. 
The total variation distance of probability measures will be denoted by $d_{TV}(\cdot,\cdot)$. 
However, we will use the notion of total variation distance not only for probability measures: 
\[d_{TV}(\mu_1, \mu_2)=\frac{1}{2}\int |f-g|d\kappa,\] where $f$ and $g$ are the density functions of $\mu_1$ and $\mu_2$ with respect to some common dominating measure $\kappa$.

The Hellinger distance $d_H$ of measures $\mu, \nu$ is defined by
\begin{equation}d_H^2(\mu, \nu)=||\mu||+||\nu||-2\gamma,\label{eq:hell}\end{equation}
where the quantity $\gamma=\int \sqrt{fg} d\kappa$ is the Bhattacharyya coefficient.

\begin{prop}\label{all:dtv}Let $X, Y$ be invariant random processes. Suppose that $\mathbb E(X_o)=\mathbb E(Y_o)=0$ and the marginals of $X$ have finite second moments.  
\begin{enumerate}[(a)]\item The following inequality holds:
\begin{align*}\bar d_2^2(X, Y)&\geq d_H^2(\mu_X, \mu_Y)\\&\geq \mathbb E(X_o^2)+\mathbb E(Y_o^2) - \sqrt{\big(\mathbb E(X_o^2)+\mathbb E(Y_o^2)\big)^2-4d_{TV}^2(\mu_X, \mu_Y)}\\ &\geq \frac{2d_{TV}^2(\mu_X, \mu_Y)}{\mathbb E(X_o^2)+\mathbb E(Y_o^2)}.\end{align*}
\item Let $(X,Y)$ be an invariant pair of random processes. Suppose that $\mathbb E(X_o)=\mathbb E(Y_o)=0$ and the marginals of $X$ and $Y$ have finite second moments. If $\mu_X$ and $\mu_Y$ are singular measures, then 
\[\mathbb E(X_oY_o)=0.\]
\end{enumerate}
\end{prop}

\medskip 

\subsubsection*{Linear factor of i.i.d.\ processes and $\bar d_2$-metric} Proposition \ref{all:dtv} implies that Gaussian factor of i.i.d.\ processes are closed in the  $\bar d_2$-metric on the $d$-regular tree. 

\begin{corollary}\label{cor:lin} Let $G=T_d$ be the $d$-regular tree. 
Suppose that $X^{(n)}$ is a sequence of Gaussian factor of i.i.d.\ processes such that $X^{(n)}\rightarrow X$ with respect to the $\bar d_2$-distance as $n\rightarrow \infty$. Then (the distribution of) $X$ is a Gaussian factor of i.i.d.\ process.
\end{corollary}

In fact, we will prove this statement for arbitrary vertex-transitive graphs, but for a smaller class of linear factor of i.i.d.\ processes (see Definition \ref{def:sphlin} and Subsection \ref{closed}).

\medskip

\subsubsection*{Gaussian wave functions.}As another application of the inequality in Proposition \ref{all:dtv}, we will show  that Gaussian wave functions are separated from factor of i.i.d.\ processes with respect to the $\bar d_2$-distance. A Gaussian wave function with eigenvalue $\lambda$ is an invariant Gaussian process $(X_v)$ satisfying 
\[(AX)_v = \lambda X_v \text{ almost surely for all } v\in V(T_d).\]
According to Theorem 4 of \cite{V}, such a process exists for all $\lambda$ in the spectrum of the adjacency operator of $G$, but it is not a factor of i.i.d. process if $\lambda$ is the supremum of the spectrum. As for the $d$-regular tree $G=T_d$ (with $d\geq 3$),  Theorem 3 of \cite*{E} states that Gaussian wave function $(X_v)$ exists for all $\lambda\in [-d,d]$. Theorem 4 of the same paper says that $(X_v)$ is a weak limit of factor of i.i.d. processes if $\lambda$ is in the spectrum of the tree, i.e. in $[-2\sqrt{d-1},2\sqrt{d-1}]$, but it is known that it is not a factor of i.i.d.  See also Corollary 3.3 of \cite{lyons}, which gives a 2-valued example for a process which is weak limit of factor of i.i.d. but not factor of i.i.d.

Proposition \ref{all:dtv} will imply the following about wave functions. 
  \begin{corollary} \label{cor:dbar} Let $X$ be a Gaussian wave function corresponding to $\lambda\in [-d,d]$ (if it exists) with $\mathrm{Var}(X_o)=1$ and $\mathbb E(X_o)=0$.
\begin{enumerate}[(a)]
\item If $X'$ is a Gaussian wave function corresponding to $\lambda'\neq \lambda$ with $\mathbb E(X'_o)=0$ and $\mathrm{Var}(X'_o)=1$, then they are orthogonal in every coupling; equivalently,
$\bar d_2(X, X')=\sqrt 2$.

\item If the spectral measure of $G$ is absolutely continuous with respect to the Lebesgue measure, and $Y$ is a limit of factor of i.i.d.\ processes in the $\bar d_2$-distance with $\mathrm{Var}(Y_o)=1$, then they are orthogonal in every coupling; equivalently, $\bar d_2(X,Y)=\sqrt 2$. 

\end{enumerate}
\end{corollary}
 
 In particular, Gaussian wave functions are never factor of i.i.d.\ processes. Moreover, the metric space of  processes with the $\bar d_2$-distance is not separable.

\medskip 

\subsubsection*{Gauss Markov processes} As an application of Theorem \ref{tfiid}, we will show the following characterization of factor of i.i.d.\ Gauss Markov processes on $G=T_d$. These are Gaussian processes which have the spatial Markov property. The covariance structure is exponential in this case. That is, for every Markov process there exist $\varrho\in \mathbb R$ such that \begin{equation}\label{eq:gm}\mathrm{cov}(X_o, X_v)=\varrho^{|v|} \qquad (v\in V(T_d)),\end{equation}
where $|v|$ denotes the distance of $v$ from the root $o$.

\begin{prop}\label{prop:GM}A Gauss Markov process is factor of i.i.d. process if and only if 
\[|\varrho|\leq \frac{1}{\sqrt{d-1}} \quad \text{ holds, where $\varrho$ is defined by equation \eqref{eq:gm}}.\]
\end{prop}

\medskip

\subsubsection*{Outline} The paper is built up as follows. In Section \ref{prel}, we recall the concept of a factor of i.i.d.\ process and spectral measures. The proofs of Proposition \ref{all:dtv}, Theorem \ref{tfiid} and Theorem \ref{tlimfiid} may be found in Sections \ref{sec:allitas}, \ref{proof2},  and \ref{proof3} respectively. The last section contains the proofs of the applications (process spectrum, $\bar d_2$-limits of linear factor of i.i.d.\ processes, Gaussian wave functions, Gaussian free field, Gauss Markov processes).

\section{Preliminaries}\label{prel}

First we recall the definition of factor of i.i.d.\ processes for vertex-transitive graphs. This section is based on \cite{V}.

\subsection{Invariant random processes and Gaussian processes.} 
Let $G$ be a rooted graph with countable vertex set $V(G)$, edge set $E(G)$ and root $o\in V(G)$. 
We consider random processes indexed by the vertices of the graph. That is, we 
assign a random variable $X_v$ to each $v\in V(G)$. 

We say that a bijection $\Phi: V(G)\rightarrow V(G)$ is an {\it automorphism} of $G$ if 
for  all vertices $u$ and $v$ the following holds: 
$(\Phi(u), \Phi (v))\in E(G)$ if and only if $(u,v)\in E(G)$. The group of 
automorphisms of $G$ is denoted by $\text{Aut}(G)$. The graph is {\it vertex 
transitive} if $\text{Aut}(G)$ acts transitively on $V(G)$. 

The collection of random variables $(X_v), v\in V(G)$ is an {\it invariant random 
process} on $G$ if for any $\Phi\in \text{Aut}(G)$ the joint distribution of 
$(X_{\Phi(v)})$ is the same as the joint distribution of $(X_v)$.

 We say that a 
collection of random variables $(X_v), v\in V(G)$ is a {\it Gaussian process}  if their joint distribution is 
Gaussian (i.e.\ any finite linear combination of them has Gaussian distribution) 
and they all have mean 0. An {\it invariant Gaussian process} is a 
Gaussian process that is invariant.

\subsection{Factor of i.i.d.\ processes.}
 
 \label{fiid}

Let $G$ be a vertex transitive graph and $\Omega=\mathbb R^{V(G)}$. Let $P$ be the product measure on $\Omega$ obtained from the 
standard normal distribution on $\mathbb R$. That is, if we assign independent standard normal 
random variables $Z_v$ to the vertices $v\in V(G)$, then the collection of random variables $(Z_v)$ is a 
random element of the measure space $(\Omega, P)$.  

A {\it factor of i.i.d.\ process} on $G$ will be determined by a function $f\in L^2(\Omega, P)$ which is 
invariant under the root-preserving automorphisms of $G$. Given the independent standard normal 
random variables $(Z_u)_{u\in V(G)}$, we assign to each vertex $v\in V(G)$ the value of 
$f$ on $(Z_{\Phi(u)})_{u\in V(G)}$ where $\Phi$ is in ${\rm Aut}(G)$ taking $v$ to the root. Notice that this does not depend on the choice of $\Phi$, because $f$ is invariant under the root-preserving automorphisms of $G$. We get an invariant random process $(X_v)_{v\in V(G)}$ this way. 
 In addition, the only possibility to define the process such that each $X_u$ is determined  from $(Z_v)$ by a 
measurable function and they are ${\rm Aut}(G)$-equivariant (they commute with the natural action of ${\rm Aut}(G)$) is to evaluate a measurable function $f$ at each vertex. 
See Section 3.2. of \cite{V} for the details. On the other hand, 
since we will deal 
with the covariance structure of the process, we need $f$ to be in $L^2(\Omega, P)$.
 
\subsubsection*{Linear factor of i.i.d.\ processes.} \label{lin} 
Linear factor of i.i.d. processes will be given by $\ell^2$-functions on the vertex set of $G$.   More precisely, first we define the $\ell^2$-space of the graph by
\[\ell^2(G)=\bigg\lbrace \alpha: V(G)\rightarrow \mathbb R \bigg\vert \sum_{v\in V(G)} \alpha(v)^2<\infty \bigg\rbrace.\]   
 In this space we have the inner product as usual:
\[\langle \alpha, \beta\rangle_G=\sum_{v\in V(G)} \alpha(v)\beta(v)\qquad (\alpha, \beta \in \ell^2(G)).\]

\begin{df}[Linear factor of i.i.d.\ process] \label{def:linfiid}We say that $X$ is a linear factor of i.i.d.\ process if its rule $f$ is given by 
\[f(\omega)= \sum_{u\in V(G)} \beta(u) \omega_u \qquad ( \omega\in \Omega),\]
where $\beta\in \ell^2(G)$ and it is invariant under the root-preserving automorphisms of $G$ (that is, $\beta(u)=\beta(\Phi(u))$ for all $\Phi\in \mathrm{Aut}(G)$ with $\Phi(o)=o$).
\end{df}

 By Kolmogorov's three-series theorem the sum in the definition is convergent almost surely if and only if $\sum_{u\in V(G)} \beta^2(u)<\infty$. 
See Proposition 3.3. of \cite{V} for more details.

\subsubsection*{Spherical linear factor of i.i.d.\ processes.}
We will also consider a special class of linear factor of i.i.d. processes. In this case the coefficients belong to the $\ell^2$-closure of the polynomials of the adjacency operator of $G$. Let the adjacency operator $A: \ell^2(G)\rightarrow \ell^2(G)$  defined by 
\[(A\beta)(v)= \sum_{(u,v)\in E(G)} \beta(u) \qquad (v\in G,\ \beta\in \ell^2(G)).\]  
We denote by $\delta_o\in \ell^2(G)$ the indicator function of the root. For each polynomial $p$ the function $p(A)\delta_o$ is a finitely supported function on $V(G)$, hence it is in $\ell^2(G)$.  Let 
\[\mathcal L=\overline{\big\{p(A)\delta_o: p \text{ is a polynomial}\big\}}\subseteq \ell^2(G),\]
that is, the $\ell^2$-closure of the functions given by the polynomials of $A$. 

\begin{df}[Spherical linear factor of i.i.d.] \label{def:sphlin}We say that $X$ is a spherical linear factor of i.i.d.\ process if its rule $f$ is given by 
\[f(\omega)= \sum_{u\in V(G)} \beta(u) \omega_u \qquad ( \omega\in \Omega),\]
where $\beta\in \mathcal L$.
\end{df}

Note that for the regular tree $G=T_d$ every finitely supported radial function on $V(G)$ (i.e., the value depends only on the distance from the root) is in the form $p(A)\delta_o$ for some polynomial $p$. Hence $\mathcal L=\ell^2(T_d)$, and every linear factor of i.i.d.\ process is spherical.

\subsubsection*{Linear factors}In the sequel, we will also consider finite linear factors of any invariant process $X$. The definition is based on polynomials of $A$, as follows.  

\begin{df}[Linear factor of a process $X$]\label{def:polylin}
Let $X$ be an invariant random process and $p$ a polynomial. By $p(A)X$ we mean the linear factor process defined by
\[[p(A)X]_v=\sum_{w\in V(G)} \big (p(A)\delta_o\big)(\Phi(w)) X_w, \]
where $\Phi\in \mathrm {Aut}(G)$ is any automorphism taking $v$ to the root $o$.
\end{df}
Since $p(A)\delta_o$ is a fixed point of root-preserving automorphisms of $G$, the definition does not depend on the choice of $\Phi$.

\subsection{Spectral measures.}  \label{plancherel} 

\subsubsection*{Spectral measure of the graph} Let $G$ be a vertex-transitive graph with all vertices having degree $d$. Recall that $\delta_o\in \ell^2(G)$ is the function that is equal to 1 at the root and 0 everywhere else. The operator $A$ is a bounded self-adjoint operator. Therefore there exists a finite measure $\nu$ on $[-d, d]$ such that 
\begin{equation}\label{eq:Ak}
\langle A^k \delta_o, \delta_o\rangle_G =\int t^k d\nu(t)
\end{equation}
holds for every $k\geq 0$. This is called the {\it spectral measure of the graph at the root} (see e.g.\ Bordenave, Sen and Vir\'ag (2013), Ab\'ert, Thom and Vir\'ag (2014) and the references therein). Notice that this is closely related to the return probabilities of the random walk on the graph: the left hand side is the number of returning paths of length $k$ starting from the root.

When $G=T_d$ is the infinite $d$-regular tree ($d\geq 3$), then $\nu$  is 
the Plancherel (or Kesten--McKay) measure\ \cite[see e.g.][]{woess}, which has density  function
\[h(t)=\begin{cases}\frac{d}{2\pi} \frac{\sqrt{4(d-1)-t^2}}{d^2-t^2}  &\qquad t\in\Big[-2\sqrt{d-1}, 2\sqrt{d-1}\Big];\\ \qquad h(t)=0  &\qquad \text{otherwise} \end{cases}\]
 with respect to the Lebesgue measure. Here $\varrho=2\sqrt{d-1}$ is the spectral radius of the tree.

 \subsubsection*{Spectral measure of invariant processes.}
 
 The spectral measure of invariant random processes will be the spectral measures of the graphings (as bounded self-adjoint operators) associated to them (see also \cite{decay}). This concept is as follows.

   Let $X$ be an invariant random 
process with marginals having finite variance. We modify  the definition of $\Omega=\mathbb R^{V(G)}$ a bit. Namely, we 
identify two elements if one can be obtained from the other by applying a root-preserving 
automorphism of $G$. We get $\tilde\Omega$ this way. The distribution of 
$X$ is a probability measure $Q$ on $\tilde\Omega$.

We  construct a bounded degree graph from $\tilde \Omega$ by connecting two elements if and only if one can be obtained from the other one by moving the root to one of its neighbors. Let $\mathcal E$ be the edge set of this graph. 

\begin{remark}
The probability space $(\tilde \Omega, Q)$ with edge set $\mathcal E$  will form a graphing.  See 
\citet*[][\!\!, Section 3]{HLSz} or \citet*[][\!\!, Section 2]{decay} for the definition and details. 
\end{remark}

We define the following operator 
$\mathcal G$ acting on $L^2(\tilde\Omega, Q)$: 
\[(\mathcal Gf)(\omega)=\sum_{(\omega,\omega')\in \mathcal E} f(\omega') \qquad (f \in L^2(\tilde\Omega, Q),\ \omega\in \tilde \Omega). \]
Notice that $\mathcal G$ is a bounded self-adjoint operator.    Let us denote by $e_o\in L^2(\tilde\Omega, Q)$ the function that 
assigns to each element of $\tilde\Omega$ its value at the root. 

According to the spectral theorem, the operator $\mathcal G$ has a (finite) spectral measure $\mu_X$ at $e_o$. That is,  $\mu_X$ is a finite measure on $\mathbb R$ satisfying the following: 
\begin{equation*}
\langle \mathcal G^k e_o, e_o\rangle_Q=\int t^k d\mu_X(t)  \qquad (k\geq 0),
\end{equation*} 
where $\langle \cdot, \cdot \rangle_Q$ is the scalar product in $L^2(\tilde\Omega, Q)$. Since the degrees in $G$  are bounded by $d$, the same holds for the graph constructed above. This implies that the largest eigenvalue (in absolute value) is equal to $d$, and the support of $\mu_X$ is contained in 
the interval  $[-d, d]$. 

\begin{df}[Spectral measure of a process] The spectral measure of an invariant random process $X$ (with marginals having finite second moments)  is the spectral measure of $\mathcal G$ at $e_o\in L^2(\tilde\Omega, Q)$, which is a finite measure on $[-d, d]$. It will be denoted by $\mu_X$.
\end{df}

 In the sequel, we will use the notation 
 \[\langle X, Y\rangle=\mathbb E(X_o \cdot Y_o)\]
 for any invariant pair of random processes $(X, Y)$ on $G$ which have marginals with finite second moments.
   
   Notice that  $X$ is a random element of $\tilde \Omega$ with distribution $Q$, the function $e_o$ selects the value at the root, and the action of $\mathcal G$ corresponds to the linear factor process given by $A$ (recall Definition \ref{def:polylin}).  We obtain 
\[\langle \mathcal G^k e_o, e_o\rangle_Q=\mathbb E ([A^k X]_o \cdot X_o)=\langle A^k X, X\rangle.\]

  We conclude that  for every invariant process $X$ its spectral measure $\mu_X$ satisfies 
\begin{equation}\label{eq:sp}
\langle A^k X, X\rangle = \int t^k d\mu_X(t) \qquad (k\geq 0), \textrm{ and}
\end{equation}
\begin{equation}\label{eq:adj}
A \textrm{ is self-adjoint: } \langle A^k X, A^l X\rangle=\langle A^{k+l} X, X \rangle \qquad (k, l\geq 0).
\end{equation}

 \subsubsection*{Covariance structure} We will define the covariance structure of a process, and we will see how it is related to the spectral measure if the marginals of the process have mean 0. 
 
 \begin{df}[Covariance structure] Let $X$ be an invariant random process with marginals having finite second moments. Its covariance structure,  $c_{X}: V(G)\rightarrow \mathbb R$ is defined by
\[c_{X}(v)=\mathrm{cov}(X_o, X_v) \qquad (v\in V(G)).\] 
\end{df}

Suppose that $\mathbb E(X_o)=0$. 
The scalar product $\langle A^k X, X\rangle$ is the same as the following: take $X$, start a random walk of length $k$ from the root, calculate the covariance of the values of $X$ at the root and at the endpoint of the random walk, and multiply this by $d^k$. On the other hand, by linearity of expectation,  we obtain 
\[\langle A^k X, X\rangle=\langle A^k\delta_o, c_X\rangle_G \qquad (k\geq 0). \] 
Notice that this is a finite sum. By equation \eqref{eq:sp} we get
 \begin{equation}\label{eq:cov}\langle A^k\delta_o, c_X\rangle_G = \int t^k d\mu_X(t) \qquad (k\geq 0),\end{equation} 
 which shows the connection between the covariance structure and the spectral measure of the process.

  Notice that if $X$ is a process that assigns independent random variables to the vertices of $G$, then $c_X=\delta_o$ and $\mu_X=\nu$; that is, the spectral measure of the process is the spectral measure of the graph.

\subsection{Limits of random processes.}

The $\bar d_2$-distance of invariant processes (Definition \ref{def:dbar}) induces the notion of $\bar d_2$-limits of factor of i.i.d.\ processes.

We will also deal with distributional convergence. That is, 
a sequence of invariant random processes converges if the sequence of their distributions 
on $\Omega=\mathbb R^{V(G)}$ converges weakly to some probability measure on $\Omega$. A weak limit of 
factor of i.i.d.\ process is an invariant process that can be approximated with factor of i.i.d.\ processes 
in distribution. Notice that $\bar d_2$-convergence implies convergence in distribution.

\section{Proof of Proposition \ref{all:dtv}}

\label{sec:allitas}

Before proving Proposition \ref{all:dtv}, we need a lemma. Recall that $p(A)X$ denotes the linear factor process of $X$ whose coefficients are given by the finitely supported function $p(A)\delta_o$ (Definition \ref{def:polylin}).

\begin{lemma}[Isometry]\label{lem:efx2}
Let $X$ be a random invariant process and $p$ a  polynomial. Assume that $\mathbb E(X_o)=0$ and the marginals of $X$ have finite second moments.  
Then we have
\[\mathbb E \big[[p(A)X]_o^2\big]=\int  p^2 d\mu_X .\]
\end{lemma}

\begin{proof}
Using the definition of the spectral measure $\mu_X$ (equation \eqref{eq:sp}) and the fact that $A$ is self-adjoint by equation \eqref{eq:adj}, we get 
\[\int  p^2 d\mu_X= \langle p^2(A)X, X\rangle=\langle p(A)X, p(A)X\rangle=\mathbb E \big[[p(A)X]_o^2\big].\qedhere\]
\end{proof}

Notice that by the inequality of arithmetic and geometric means, the Bhattacharyya coefficient can be expressed as 
\[2\gamma=\inf_{\substack{h> 0\\ h, 1/h \text{ integrable}}} \int h d\mu +\int \frac{1}{h} d\nu;\]
or, equivalently, using polynomials: 
\begin{equation}2\gamma=\lim_{\varepsilon \downarrow 0} \inf_{ \substack{p,q \text{ polynomials}\\ |pq-1|<\varepsilon \text{ on } \mathrm{supp}(\kappa)}} \bigg\{\int p^2 d\mu +\int q^2 d\nu \bigg\}.\label{hell}\end{equation}

\begin{proof}[{\bf Proof of Proposition  \ref{all:dtv}.}] 

Take any invariant coupling of $X, Y$. Let $\kappa=\mu_X+\mu_Y$, and $d\mu_X=fd\kappa$, $d\mu_Y=gd\kappa$.
By Lemma \ref{lem:efx2} and equation \eqref{eq:hell}, we need to show $|\mathbb E(X_oY_o)|\leq \gamma$. For polynomials $p, q$ as in equation \eqref{hell}, we have 
\begin{align*}\mathbb E\big((p(A)X)^2_o\big)&+\mathbb E\big((q(A)Y)^2_o\big)\geq 2\big|\mathbb E(p(A)X_o\cdot q(A)Y_o)\big|\\&=2\big|\mathbb E(q(A)p(A)X_o\cdot Y_o)\big|=2\big|\mathbb E(X_oY_o)\big|+r, \end{align*}
where $|r|\leq \varepsilon \sqrt{\mathbb E(X_o^2)\mathbb E(Y_o^2)}$ holds for the error. We conclude by letting $\varepsilon\rightarrow 0$. 

The second inequality is standard, see e.g.\ Lemma 12.2.\ in \cite{hell}. We include a proof for completeness.
We define 
\[\Delta_1=\int |f-g|d\kappa; \qquad s=\int (f+g) d\kappa. \]
By the Cauchy--Schwarz inequality, we have 
\[\int (\sqrt f- \sqrt g)^2d\kappa \int (\sqrt f+ \sqrt g)^2d\kappa \geq \bigg [\int |f-g| d\kappa \bigg]^2=\Delta_1^2.\]
On the other hand, the definition of $s$ implies that
\[\int (\sqrt f+ \sqrt g)^2d\kappa=\int f+g+2\sqrt{fg}d\kappa=2s-\int (\sqrt f- \sqrt g)^2d\kappa.\] 
Hence for $q=\int (\sqrt f- \sqrt g)^2d\kappa$ we obtain $q(2s-q)\geq \Delta_1^2$.
It follows that 
$q\geq s-\sqrt{s^2-\Delta_1^2}$.

The last inequality is straightforward.

As for $(b)$, singularity of $\mu_X$ and $\mu_Y$ implies that $\int \sqrt{fg} d\eta=0$. Hence the statement follows from the  above inequality $|\mathbb E(X_oY_o)|\leq \gamma$.
\end{proof}

\section{Proof of Theorem \ref{tfiid}}

\label{proof2}

\begin{lemma}[Spectral measures of spherical linear functions] \label{lem:gyok}
There exists an isometry between $\mathcal L$ and $L^2(\mathbb R, \nu)$ such that the following hold.
\begin{enumerate}[(a)] 
\item For all 
$\alpha \in\mathcal L$ and its image $\hat\alpha$ we have 
\[\langle\alpha, \alpha\rangle_G=\int \hat \alpha^2 (t)d\nu(t).\]
\item Let $\alpha\in \mathcal L$. Then the spectral measure of the spherical linear factor of i.i.d.\ process given by $\alpha$ has density $\hat \alpha^2$ with respect to $\nu$.
\item Let $h\in L^2(\mathbb R, \nu)$ be a measurable function. Then there exists a spherical linear factor of i.i.d.\ process $X$ such that $\mu_X$ is the measure with density function $h^2$ with respect to $\nu$.
\end{enumerate}
\end{lemma}

\begin{proof} 
$(a)$ Let $p$ be a polynomial. By equation \eqref{eq:Ak} we have 
\[\langle p(A)\delta_o, p(A)\delta_o\rangle_G=\int  p^2(t)d\nu(t).\]  
It follows that $p(A)\delta_o\mapsto p$ is an isometry from the appropriate subspace of $\ell^2(G)$  to the set of polynomials in $L^2(\mathbb R, \nu)$.  As usual, this isometry has a unique extension to the $\ell^2$-closure; that is, there is an isometry between $\mathcal L$ and $L^2(\mathbb R, \nu)$ (as polynomials form a dense set in the latter space).

$(b)$ Let $(p_n)$ be a sequence of polynomials tending to $\hat \alpha$ in $L^2(\mathbb R, \nu)$. Let $Y$ be an i.i.d.\ process, and  $X$ is the spherical linear factor of i.i.d.\ process obtained by $\alpha$ from $Y$. The Cauchy--Schwarz inequality implies
\begin{equation*}\lim_{n\rightarrow \infty} \langle A^k p_n(A)Y, p_n(A)Y\rangle=\langle A^k X, X\rangle \qquad (k\geq 0).\end{equation*}
By equations \eqref{eq:Ak} and \eqref{eq:cov} this yields 
\begin{equation*}\lim_{n\rightarrow \infty} \int t^k p_n^2(t) d\nu(t)=\lim_{n\rightarrow \infty} \langle A^k p_n(A)\delta_o, p_n(A)\delta_o\rangle_G=\langle A^k X, X\rangle \ (k\geq 0).\end{equation*}

On the other hand, since $(p_n)$ tends to $\hat \alpha $ in $L^2(\mathbb R, \nu)$, we conclude that 
\[\langle A^k X, X\rangle=\int t^k \hat\alpha^2 (t) d\nu(t) \qquad (k\geq 0).\]
This implies that $d\mu_X=\hat\alpha^2 d\nu$, and the proof of part $(b)$ is complete.

$(c)$ Based on the isometry between $\mathcal L$ and $L^2(\mathbb R, \nu)$, let $\alpha \in \mathcal L$ be chosen such that $\hat \alpha = h$. Let $X$ be the spherical linear function given by $\alpha$. The argument in the proof of $(b)$ implies that the spectral measure of $X$ has density $h^2$ with respect to $\nu$.
\end{proof}

\begin{lemma}Let $X$ be an invariant random process such that $\mathbb E(X_o)=0$, $X$ has marginals with finite second moments, and $c_X\in \ell^2(G)$. Then $\mu_X$ is absolutely continuous with respect to the spectral measure $\nu$ of the graph.  \label{lem:block}
\end{lemma}
\begin{proof}
Suppose first that $c_X=p(A)\delta_o$ for some polynomial $p$. Then, by self-adjointness and equation \eqref{eq:Ak}, we have  
\[\langle A^k \delta_o, c_X\rangle_G=\langle A^k \delta_o, p(A)\delta_o\rangle_G=\int t^k p(t) d\nu(t)\qquad (k\geq 0).\]
By equation \eqref{eq:cov} we obtain that the measure with density function $p$ will be the spectral measure of $X$. Hence $\mu_X$ is absolutely continuous with respect to $\nu$.

Next, we assume that $c_X\in \mathcal L$. Let $(p_n)$ be a sequence of polynomials converging to $c_X$ in $\ell^2(G)$. Taking limits in the previous equation and applying Cauchy--Schwarz inequality we obtain that 
\begin{align*}\langle A^k \delta_o, c_X\rangle&= \lim_{n\rightarrow\infty} \langle A^k \delta_o, p_n(A)\delta_o\rangle_G =\lim_{n\rightarrow\infty} \int t^k p_n(t) d\nu(t)\\&=\int t^k \widehat{c_X} d\nu(t),\end{align*}
where $\widehat{c_X}\in L^2([-d, d], \nu)$ is the image of $c_X$ at the isomorphism defined in Lemma \ref{lem:gyok}. Again, this implies that $\mu_X$ is absolutely continuous with respect to $\nu$ with density $\widehat{c_X}$. 

Finally, let $c_X\in \ell^2(G)$ be arbitrary. By definition, $\mathcal L$ is a closed linear subspace in the Hilbert space $\ell^2(G)$. Let $\overline{c_X}\in \mathcal L$ be the projection of $c_X$ into this subspace. Since the projection does not change the scalar product with vectors from the subspace, we have 
\[\langle A^k \delta_o, c_X\rangle_G=\langle A^k \delta_o, \overline{c_X}\rangle_G \qquad (k\geq 0).\]
  Therefore the spectral measure of $X$ is the same as the spectral measure corresponding to $\overline {c_X}$, which is absolutely continuous with respect to $\nu$.
  \end{proof}

\begin{proof}[{\bf Proof of Theorem $\ref{tfiid}$.}]

We start  with showing that $(i)$ implies $(ii)$. Let $\mu$ be a finite measure, which is absolutely 
continuous with respect to the spectral measure $\nu$; its density is $g\geq 0$.  Since $\sqrt{g}\in L^2([-d,d], \nu)$,  by part $(c)$ of Lemma \ref{lem:gyok} there exists a spherical linear factor of i.i.d.\ process with spectral measure having density $g$ with respect to $\nu$.

$(ii)\Rightarrow (iii)$, $(iii)\Rightarrow (iv)$ are trivial.

$(iv)\Rightarrow (v)$ We say that a factor of i.i.d.\ process is a block factor of i.i.d.\ if the factor map depends only on the values in a ball of finite radius.  

A factor of i.i.d.\ process can be approximated with block factor of i.i.d.\ processes in the $\bar d_2$-distance. See \cite{lyons} for trees, but the proof is similar for vertex-transitive graphs. (To see this, let $f\in L^2(\Omega, P)$ be the rule of a factor of i.i.d.\ process, and let $\mathcal F_r=\sigma\{\omega_v: |v|\leq r\}\subset \Omega$. By the martingale convergence theorem, the sequence of block factor of i.i.d.\ processes with rule $\mathbb E(f|\mathcal F_r)$ converges to the factor of i.i.d.\ process with rule $f$ in the $\bar d_2$-metric.) Thus,  $\bar d_2$-limits of factor of i.i.d.\ processes can also be approximated with block factors. The covariance structure of a block factor of i.i.d.\ process is finitely supported, hence it is in $\ell^2(G)$. Thus every block factor of i.i.d.\ process has absolutely continuous spectral measure by Lemma \ref{lem:block}. This shows that $(iv)$ implies $(v)$.

Finally, we prove that $(v)$ implies $(i)$. Let  $(X^{(n)})_{n\geq 1}$ be a sequence of invariant random processes such that each $\mu^{(n)}=\mu_{X^{(n)}}$ is absolutely continuous  with respect to the spectral measure $\nu$ with density function $g^{(n)}$. Furthermore, assume that $X^{(n)}\rightarrow X$ in the $\bar d_2$-distance as $n\rightarrow \infty$.

It follows from the Cauchy--Schwarz inequality and the definition of $\bar d_2$-convergence that $\mathrm{Var}(X^{(n)}_o)\rightarrow \mathrm{Var}(X_o)$ holds as $n\rightarrow \infty$. 
Therefore Proposition \ref{all:dtv} implies that
\begin{equation}\label{eq:tv}d(\mu^{(n)},\mu_X)_{TV}\rightarrow 0 \qquad (n\rightarrow\infty).\end{equation}
We have
\[d\big(\mu^{(n)},\mu^{(m)}\big)_{TV}=\frac{1}{2}\int \big|g^{(n)}-g^{(m)}\big|d\nu \qquad (n,m\geq 1).\]
Equation \eqref{eq:tv} implies that $\big(g^{(n)}\big)$ is a Cauchy sequence in $L^1([-d,d], \nu)$.  Due to the completeness of this space, we get that there exists $g\in L^1([-d,d], \nu)$ such that 
\[g^{(n)}\rightarrow g\quad \text{ in }L^1([-d,d], \nu) \qquad (n\rightarrow \infty).\]
The dominated convergence theorem implies that \[\mu^{(n)}=g^{(n)}d\nu\stackrel{TV}\rightarrow g d\nu \qquad (n\rightarrow \infty)\] with respect to the total variation distance 
(by this notation, we mean the measures having the given density with respect to $\nu$). Hence $d\mu_X=g d\nu$, and we conclude that $\mu_X$ is absolutely continuous with respect to the spectral measure of the graph $G$.
\end{proof}

\section{Proof of Theorem \ref{tlimfiid}} 

\label{proof3}

\begin{lemma}\label{lem:limit}Let $(X^{(n)})$ be a sequence of invariant random processes (with marginals having finite second moments) with covariance structures $(c_n)$ and spectral measures $(\mu_n)$. 
\begin{enumerate} [(a)]
\item  Suppose  that $X$ is another invariant random process such that
\[c_n(v)\rightarrow c_X(v) \text{ for all } v\in V(G) \text{ as } n\rightarrow \infty.\] 
Then $\mu_n\rightarrow \mu_X$ weakly as $n\rightarrow\infty$. In particular, $\mu_n$ tends to $\mu$ weakly if $X^{(n)}$ tends to $X$ in distribution and each $X^{(n)}$ is a Gaussian process.

\item Suppose that $\mu_n$ converges weakly to a finite measure $\mu$. Then there exists an invariant random process $X$ with spectral measure $\mu=\mu_X$. 
\end{enumerate}
\end{lemma}

\begin{proof}
$(a)$ Since every $p(A)\delta_o$ has finite support, by equation \eqref{eq:Ak} we obtain that
\[\int p(t) d\mu_n(t)=\langle p(A) \delta_o, c_n\rangle_G\rightarrow \langle p(A) \delta_o, c\rangle_G=\int p(t) d\mu(t)\]
holds as $n\rightarrow \infty$ for all polynomials $p$. The set of polynomials is dense in the set of continuous functions with respect to the supremum norm, hence  $\mu_n$ tends to $\mu$ weakly.

Furthermore, in case of Gaussian processes, convergence in distribution implies the pointwise convergence of the covariance structures.

$(b)$ We neglect the degenerate case $\mu([-d, d])=0$. Lemma \ref{lem:efx2} and the finiteness of $\mu$ implies that $\mathbb E\big(\big[X^{(n)}_o\big]^2\big)$ is bounded.  Hence the covariance structures are uniformly bounded by the Cauchy--Schwarz inequality. Since the vertex set of $G$ is countable, we may choose a subsequence $X^{(n_m)}$ such that 
\begin{equation}\label{eq:nm}c_{n_m}(v)\rightarrow c(v) \text{ for all } v\in V(G) \text{ as } m\rightarrow \infty\end{equation}
for some $c: V(G)\rightarrow \mathbb R$. Pointwise convergence preserves the property of being a covariance structure of an invariant process, hence there exists an invariant Gaussian process $X$ with covariance structure $c$. On the other hand, equation \eqref{eq:nm} and part $(a)$ imply that  $\mu_{n_m}\rightarrow \mu_X$ weakly as $m\rightarrow \infty$. We conclude that $\mu=\mu_X$.
\end{proof}

\begin{proof}[Proof of Theorem $\ref{tlimfiid}$]

$(ii)\Rightarrow (iii)$ is trivial.

$(iii)\Rightarrow (i)$ 
Suppose that the invariant process $X$ is the limit of the factor of i.i.d.\ processes $(X^{(n)})$ in distribution, and $X$ has marginals with finite second moments.  Denote by $X^{(n,T)}$ the truncation of $X^{(n)}$ at level $T$:
\[X^{(n,T)}_v=\begin{cases} T, &\text{ if } X^{(n)}_v>T;\\ X^{(n)}_v, &\text{ if } X^{(n)}_v\in [-T, T];\\-T, &\text{ if } X^{(n)}_v<-T; \end{cases} \qquad (v\in V(T_d)).\]  Let $X^T$ be the truncated version of $X$.

Since the truncation is a continuous function, $X^{(n,T)}\rightarrow  X^{T}$ in distribution as $n\rightarrow \infty$. By boundedness, we have 
\begin{equation}\label{eq:cov1}
\mathrm{cov}(X^{(n,T)}_v,X^{(n,T)}_w)\rightarrow \mathrm{cov}(X^{T}_v,X^{T}_w) \qquad (n\rightarrow \infty)
\end{equation}
for all $v,w\in V(T_d)$ and $T>0$. 

The second moment condition on $X$ implies 
\begin{equation}\label{eq:cov2}
\mathrm{cov}(X^{T}_v,X^{T}_w)\rightarrow \mathrm{cov}(X_v,X_w) \qquad (T\rightarrow \infty)
\end{equation}
for all $v,w\in V(G)$.

Since $X^{(n)}$ is a factor of i.i.d.\ process, $X^{(n,T)}$ is also factor of i.i.d. 
By \eqref{eq:cov1}, \eqref{eq:cov2} and a diagonalization argument we obtain that the covariance structure of $X$ is the pointwise limit of covariance structures of factor of i.i.d. processes. (Note that $\mathrm{cov}(X_v^{(n)}, X_w^{(n)})$ may not converge to $\mathrm{cov}(X_v, X_w)$.)

To summarize, we can find a sequence of factor of i.i.d.\ covariance structures  $(c_m)$ such that $c_m(v)\rightarrow c_X(v)$ for all $v\in V(G)$ as $m\rightarrow \infty$.  Part $(a)$ of Lemma \ref{lem:limit} implies that $\mu_m$ tends to $\mu$ weakly. By Theorem \ref{tfiid} every $\mu_m$ is absolutely continuous with respect to $\nu$. Therefore the support of $\mu$ is contained in $\mathrm{supp}( \nu)$.

$(i)\Rightarrow (ii)$  Let $\mu$ be a finite measure such that $\mathrm{supp}(\mu)\subseteq \mathrm{supp} (\nu)$.
There exist a sequence of finite measures $(\mu_n)$ such that it converges weakly to $\mu$ and every $\mu_n$ is absolutely continuous with respect to $\nu$. 

On the other hand, Theorem \ref{tfiid} implies that for every  $\mu_n$ there exists a spherical linear factor of i.i.d process $X^{(n)}$ 
whose spectral measure is $\mu_n$.  It follows from part $(b)$ of Lemma \ref{lem:limit}  that $\mu$ is the spectral measure of some Gaussian process $X$. In addition, the proof of the lemma shows that the covariance structures of a subsequence of $X^{(n)}$ converge pointwise to the covariance structure of $X$. Since every $X^{(n)}$ is spherical, it is also Gaussian. In case of Gaussian processes the pointwise convergence of covariance structures implies convergence in distribution. Therefore some subsequence of  $\big(X^{(n)}\big)$ converges to $X$ in distribution, which concludes the proof. 
\end{proof} 
 
\begin{remark} Notice that linear factor of i.i.d.\ processes are always Gaussian. Therefore for an absolutely continuous measure we can find a Gaussian factor of i.i.d.\ process with this spectral measure. Similarly, for measures supported on $\mathrm{supp}(\nu)$, the appropriate process is the weak limit of Gaussian factor of i.i.d.\ processes as well.
\end{remark} 
 
\section{Applications}

\label{branching}

\subsection{Process spectrum} 
\begin{proof}[Proof of Theorem \ref{tpsp}]
First we prove the statement for atomic measures. Suppose that $x\in \mathrm{psp}(G)$, and $\mu$ is an atom at $x$. Given $\varepsilon>0$, there exists a process $X^{\varepsilon}$ such that $\mu_{X^{\varepsilon}}([x-\varepsilon, x+\varepsilon])>0$. It follows that there exists a polynomial $p_{\varepsilon}$ such that the following hold for the measure $d\mu^*_{\varepsilon}=p_{\varepsilon} d \mu_{X^{\varepsilon}}$ (i.e.\ the measure $\mu^*_{\varepsilon}$ that has density function $p_{\varepsilon}$ with respect to $\mu_{X^{\varepsilon}}$): 
\[\mu^*_{\varepsilon}(\mathbb R\setminus [x-\varepsilon, x+\varepsilon])<\varepsilon; \quad \mu(x)-\varepsilon\leq \mu^*_{\varepsilon}([x-\varepsilon, x+\varepsilon])\leq \mu(x)+\varepsilon.\]
  Notice that $\mu^*_{\varepsilon}$ tends to $\mu$ weakly as $\varepsilon\rightarrow 0$.   Equation \eqref{eq:sp} implies that the process $p_{\varepsilon}(X^{\varepsilon})$, which is a finite linear factor of $X^{\varepsilon}$ by Definition \ref{def:polylin}, has spectral measure $\mu^*_{\varepsilon}$. Putting this together with   Lemma \ref{lem:limit} $(b)$, we conclude that $\mu$ is the spectral measure of some invariant random process.

Further on, if the support of $\mu$ consists of finitely many atoms, then one can take the sum of independent copies of  invariant random processes constructed for atomic measures,  multiplied by appropriate constants. 
The covariance structure is additive due to independence, which shows by equation \eqref{eq:cov} that the spectral measure is $\mu$.
  
Finally, if we have an arbitrary finite measure $\mu$ whose support is contained in $\mathrm{psp}(G)$, 
then it can be approximated weakly with measures $\mu_n$, where each $\mu_n$ is supported on finitely many atoms. For 
every $\mu_n$ we already have a process $X^{(n)}$ with spectral measure $\mu_n$. We finish the proof by applying part $(b)$ of Lemma \ref{lem:limit}. 
\end{proof}

\medskip

\begin{proof}[Proof of Proposition \ref{kazhdan}]
$(ii)\Rightarrow (i)$ We will use the next formulation of Kazhdan's property: every sequence of positive definite functions on $G$ that converges to 1 pointwise (i.e.\ on compact subsets) converges to 1 uniformly on $G$. 

Let $X$ be the following process: constant 1 on all vertices with probability $1/2$, and constant $-1$ with probability $1/2$. Then $c_X(v)=1$ for all $v\in V(G)$, and, by equation \eqref{eq:cov} it follows that $\mu_X=\delta_d$ is an atomic measure at $d$. 

Let $c_n: V(G)\rightarrow \mathbb R$ be a sequence of positive definite functions converging pointwise to $1$. We can find a sequence of invariant  processes $(X^{(n)})$ such that $X^{(n)}$ has covariance structure $c_n$.  By Lemma \ref{lem:limit} $(a)$ we obtain that $\mu_{X^{(n)}}\rightarrow \mu_X$ weakly. It follows from $(ii)$ and the fact $\mu_X=\delta_d$ that 
\[\mu_{X^{(n)}} = (1-\varepsilon_n)\delta_d + \mu'_n,\]
where $\varepsilon_n\rightarrow 0$, and $(\mu'_n)$ is a sequence of measures such that $\mu'_n([-d, d])$ tends to $0$ as $n\rightarrow \infty$, and $\mu'_n$ is supported on $\mathrm{psp}(G)$. 

Since $\mu_{X^{(n)}}$ is the spectral measure corresponding to $X^{(n)}$, we get that $\mu'_n$ is the spectral measure corresponding to $c_n-(1-\varepsilon)$. That is, we have 
\[\langle A^k(c_n-(1-\varepsilon)), \delta_o\rangle=\int t^k d\mu'_n \qquad (k=0,1, \ldots). \]
It follows  that $(c_n-(1-\varepsilon))$ is positive definite for each $n$, and hence it is a covariance structure of an invariant process. The fact $\mu'_n([-d, d])\rightarrow 0$ implies that $(c_n-(1-\varepsilon))(0)$ tends to 0. By the Cauchy--Schwarz inequality we obtain that $c_n-(1-\varepsilon)$ converges to 0 uniformly. Hence $c_n$ converges to constant 1 uniformly, and $G$ has Kazhdan's property $(T)$.

$(i)\Rightarrow (ii)$  
Let $X$ be the constant 1 process. Its spectral measure is an atomic measure at $d$, hence $d$ is always in the process spectrum. Suppose (for contradiction) that $d$ is not an isolated point. Choose a sequence of numbers $(a_n)$ with limit $d$ such that $a_n\in \text{psp}(G)$ for all $n$. Theorem \ref{tpsp} implies that we can find a sequence of invariant processes $X^{(n)}$ such that the spectral measure of $X^{(n)}$ is an atomic measure at $a_n$ for each $n$. It follows that $\mu_{X^{(n)}}\rightarrow \mu_X$ weakly as $n\rightarrow \infty$. Similarly to the proof of Lemma \ref{lem:limit}, we can assume that the covariance structures of $X^{(n)}$ converge pointwise to $c_X$, by choosing an appropriate subsequence.
By the characterization of Kazhdan's property using positive semidefinite functions (covariance structures are positive semidefinite), we get that $c_{X^{(n)}}\rightarrow c_X\equiv 1$ uniformly on $V(G)$.

On the other hand, if $H$ is infinite and has Kazhdan's property $(T)$, then it can not be amenable \citep[see e.g.][]{harpe}. Therefore its spectral radius is strictly less than $d$. Therefore the covariance of $X_o$ and $X_v$, where $v$ is the endpoint of a random walk of length $k$, decays exponentially as a function of $k$. This contradicts the uniform convergence of correlation structures above. Hence  the process spectral radius is less than $d$.\end{proof}

\subsection{Gaussian wave functions.}
\begin{proof}[Proof of Corollary \ref{cor:dbar}] We check that $\mu_{X}$ is the Dirac measure which puts an atom of weight 1 at $\lambda$:
\[\langle A^k X, X\rangle=\lambda^k \langle X, X\rangle=\lambda^k \mathrm{Var}(X_o)=\lambda^k=\int t^k d\mu_X(t).\] 

Different Dirac measures are singular measures.
On the other hand, the measure $\mu_Y$ is absolutely continuous with respect to the spectral measure of the graph by Theorem \ref{tfiid}. Hence it is absolutely continuous with respect to the Lebesgue measure in case $(b)$, and again, it is singular to the Dirac measure. Therefore both parts of the statement follow from Proposition \ref{all:dtv} $(b)$. 
 \end{proof}

\subsection{Linear factor of i.i.d.\ processes and $\bar d_2$-convergence.} 

\label{closed} 

First we prove the following statement about spherical linear processes. 

\begin{prop}\label{all:sph} 
Suppose that $X^{(n)}$ is a sequence of spherical linear factor of i.i.d.\ processes such that $X^{(n)}\rightarrow X$ with respect to the $\bar d_2$-distance as $n\rightarrow \infty$. Then (the distribution of) $X$ is a spherical linear factor of i.i.d.\ process.
\end{prop}

\begin{proof}
Let $\alpha_n\in \mathcal L$ be the function which defines $X^{(n)}$.  Lemma \ref{lem:gyok} implies     $\mu_{X^{(n)}}=\widehat {\alpha_n}^2d\nu$. On the other hand, due to the $\bar d_2$-convergence of ${X^{(n)}}$ and part $(a)$ of Proposition \ref{all:dtv}, we get  that $\mu_{X^{(n)}}$ is a Cauchy sequence in total variation distance. Therefore, since $L^1([-d, d], \nu)$ is complete, we can find a function which is the limit of the sequence $\big(\widehat {\alpha_n}^2\big)$. It has to be nonnegative, hence there exists $\hat \alpha\in L^2([-d,d], \nu)$ such that $\hat \alpha\geq 0$ and 
\[\int\big|\widehat {\alpha_n}^2 - \hat \alpha^2\big| d\nu\rightarrow 0 \qquad (n\rightarrow \infty).\]

Using the inequality $(a-b)^2\leq |a^2-b^2|$ for $a, b>0$, this yields 
\[\int (\widehat {\alpha_n}-\hat \alpha)^2d\nu\rightarrow 0 \qquad (n\rightarrow \infty).\]

Let $\alpha\in \mathcal L$ be the function corresponding to $\alpha$  according to part $(a)$ of  Lemma \ref{lem:gyok}. Let $X^{\alpha}$ be the spherical linear factor of i.i.d.\ process given by $\alpha$. By applying $\alpha_n$ and $\alpha$ on the same i.i.d. process (which defines the coupling), and using Lemma \ref{lem:gyok} $(b)$, we have
\[\int (\widehat {\alpha_n}-\hat \alpha)^2d\nu=\mathbb E\big[(X^{(n)}_o-X^{\alpha}_o)^2\big].\] 

Hence $X^{(n)}$ converges to $X^{\alpha}$ in the $\bar d_2$-distance. Convergence in   $\bar d_2$-distance implies convergence in distribution, and the limit is unique. We conclude that the distribution of $X$ is equal to the distribution of $X^{\alpha}$, which is a spherical linear factor of i.i.d.\ process.
\end{proof}

\begin{proof}[Proof of Corollary \ref{cor:lin}] Invariant Gaussian processes are determined by their covariance structure. Hence every Gaussian factor of i.i.d.\ process is a linear factor of i.i.d.\ process by Theorem \ref{tfiid}. 

On the other hand, every linear factor of i.i.d.\ process is spherical on the $d$-regular tree, because every finitely supported radial function is a polynomial of $A$.  Proposition \ref{all:sph} implies that the limit of the sequence is a linear factor of i.i.d.\ process. Hence it is Gaussian. 
\end{proof}

\subsection{Gaussian free field}

We recall the definition of an analogue of the Gaussian free field on transient graphs, and we will show that all these random processes are factor of i.i.d. processes (see also Exercise 10.31.\ in \cite{lyonsperes}). 

\begin{df}[Gaussian free field] Let $G$ be a transient and vertex-transitive graph. An invariant Gaussian random process is a Gaussian free field if its covariance structure is the Green function: 
\[\sum_{k=0}^{\infty} \bigg(\frac Ad\bigg)^k.\]
\end{df}
Notice that our assumption that $G$ is transient  implies that this series is convergent, because the number of visits of a random walk starting from the root is finite almost surely.

\begin{prop}A Gaussian free field is  a linear factor of i.i.d.\ process.
\end{prop}
\begin{proof}  Consider the function defined by \[\sum_{k=0}^{\infty} \bigg (\frac xd\bigg)^k =\frac{d}{d-x} \qquad (x\in (-d, d)).\]
The transitivity of $G$ implies that this function is in $L^1([-d,d], \nu)$, where $\nu$ is the spectral measure of $G$.  It follows from the definition of the Gaussian free field that this is the density function of its spectral measure.  Theorem \ref{tfiid} and the fact that Gaussian processes are determined by their covariance structures imply the statement.
\end{proof}

\subsection{Gaussian Markov processes on the tree}  
We will need the following family of polynomials, which plays an important role in understanding the radial functions on the regular tree. Let $(r_n)$ be the unique sequence of polynomials satisfying the following recurrence equations: 
\begin{equation}\label{rekurzio}\begin{split}
r_0(x)&=1; \\ xr_0(x)&=r_1(x); \\ xr_n(x)&=(d-1)r_{n-1}(x)+ r_{n+1}(x) \qquad (n\geq 1).\end{split}
\end{equation} 
The polynomial $r_n$ has degree $n$. These are sometimes called Dunau polynomials, and they are closely related to Chebyshev polynomials of the second kind (see e.g.\ Alon, Benjamini, Lubetzky and Sodin 2007, Backhausz, Szegedy and Vir\'ag 2014, Arnaud and Letac 1984, Fig\`a-Talamanca and Nebbia 1991). On the other hand, they are orthogonal with respect to the Plancherel measure. Moreover, the following can be proved by induction: 
$r_0(A)\delta_o=\delta_o$, and for $n\geq 1$ we have 
\begin{equation}\label{rk}[r_n(A)\delta_0](v)=\begin{cases} 1, & \text{ if } |v|=n; \\ 0, & \text{ otherwise.}\end{cases}\end{equation} 
This implies that all radial functions on the tree are given by limits of polynomials of $A$.

\begin{proof}[Proof of Proposition $\ref{prop:GM}$]
Let $c_{\varrho}: V(T_d)\rightarrow \mathbb R$ be defined by $c_{\varrho}(v)=\varrho^{|v|}$.  By equation $\eqref{eq:gm}$, this is the covariance structure of the Gaussian  Markov process.
First we decide whether $c_{\varrho}$ is in $\ell^2(T_d)$:
\[1+\sum_{k=1}^{\infty} d(d-1)^{k-1}\varrho^{2k}<\infty \Leftrightarrow |\varrho|<\frac{1}{\sqrt{d-1}}.\]
Hence in the case
$|\varrho|<\frac{1}{\sqrt{d-1}}$ Lemma \ref{lem:block} applies. However,  we compute this density function for $|\varrho|<\frac{1}{\sqrt{d-1}}$, in order to deal with the case $|\varrho|=\frac{1}{\sqrt{d-1}}$. Equations \eqref{eq:gm} and \eqref{rk} imply that 
\[c_{\varrho}=\sum_{k=0} \varrho^k r_k(A) \delta_o.\]

Based on the proof of Lemma \ref{lem:block}, we obtain that the density function (with respect to the spectral measure of the tree) of the spectral measure of the Gaussian Markov process with parameter $\varrho$ is the following:
\[f_{\varrho}(x)=\sum_{k=0}^{\infty} \varrho^k r_k(x).\]
To compute this sum, let 
\[g(x,y)=\sum_{k=1}^{\infty} r_k(x) y^k.\] 
Then by recurrence equation \eqref{rekurzio} we have
\begin{align*}
xg(x,y)&=(d-1)\sum_{k=1}^{\infty} r_{k-1}(x) y^k+\sum_{k=1}^{\infty} r_{k+1}(x) y^k\\
&= y(d-1)\sum_{k=0}^{\infty} r_k(x) y^k+\frac1y\sum_{k=2}^{\infty} r_k(x) y^k\\
&=y(d-1)r_k(0)+y(d-1)g(x,y)+\frac 1y[g(x,y)-r_1(x)y]\\
&=y(d-1)+y(d-1)g(x,y)+\frac 1yg(x,y)-x.\end{align*}

This yields 
\[g(x,y)=\frac{xy-y^2(d-1)}{1+y^2(d-1)-xy}.\] 

Hence the density function of $\mu$ is the following: 
\[f_{\varrho}(x)=g(x, \varrho)+1=\frac{1}{1+\varrho^2(d-1)-x\varrho} \qquad \Big(|\varrho|<\frac{1}{\sqrt{d-1}}\Big).\]
 
For $\varrho=\frac{1}{\sqrt{d-1}}$ the covariance structure is not in $\ell^2(T_d)$. However, 
the calculation above works and we get  
\[f(x)=\sum_{k=0}^{\infty} r_k(x)(d-1)^{-k/2}=\frac{1}{2-\frac{x}{\sqrt{d-1}} }.\]
This has a singularity only at the endpoint of the Plancherel interval, namely, at $2\sqrt{d-1}$. 
Since the density function of the Plancherel measure behaves like $\sqrt{x}$ at the endpoints of 
its support, this function is still integrable with respect to the Plancherel measure: $f\in L^1([-d,d], \nu)$. 
Therefore this is the density function of the spectral measure of the Gauss Markov process with respect to $\nu$. Hence $c_{\varrho}$ is the covariance structure of a factor of i.i.d.\ process if $|\varrho|\leq 1/\sqrt{d-1}$.
 Similar argument works for $\varrho=-1/\sqrt{d-1}$.) For Gauss Markov processes this implies that the process is linear factor of i.i.d.\ itself, according to Theorem \ref{tfiid}.  
 
 As for the other direction, we refer to \cite*{decay}: from that result it follows immediately that covariance structure with larger absolute value of $\varrho$ can not be factor of i.i.d. 
 \end{proof}
 
 \subsection{Branching Markov chains}
 We can also  examine the covariance structures of branching Markov chains on the $d$-regular tree. Fix a reversible 
Markov chain with finite state space $S$ and transition matrix $M$.  
Choose the state of the root $o$ uniformly at random. Then the Markov chain spreads out: 
the neighbors of the root get their states given the state of the root and according 
to the transition probabilities given by $M$. The transitions are conditionally independent given the 
state of the root. This is continued to get the states of the neighbors of 
the neighbors of the root, and so on. This will be an invariant random process on $T_d$. One can get the Potts and the Ising model 
as particular cases \citep[see e.g.][]{evans, sly}. 

Let $\varphi$ be eigenvector of $M$ corresponding to the largest eigenvalue. Then the correlation of $\varphi(X_o)$ and $\varphi(X_v)$ is  
$\varrho^{|v|}$, where $\varrho$ is 
the spectral radius of the transition matrix $M$. Therefore the calculation above implies that this covariance structure  is a factor of i.i.d. covariance structure if and only if $|\varrho|\leq 1/\sqrt{d-1}$. However, this does not imply that the process itself is factor of i.i.d.\ in this case. 

 A particular case is the Ising model, where $S=\{-1, 1\}$ and
\[M=\left(\begin{array}{cc} \frac{1+\varrho}{2} &  \frac{1-\varrho}{2}\\  \frac{1-\varrho}{2} &  \frac{1+\varrho}{2}\end{array}\right).\] 
It is known that the process is itself a factor of i.i.d.\ if $|\varrho|\leq \frac{1}{d-1}$, see e.g.\ \cite{lyons}. It is open whether the Ising model itself is factor of i.i.d.\ in the  case $1/(d-1)<|\varrho|\leq 1/\sqrt{d-1}$.

 \subsection{Open questions} We finish the paper with some open questions. 
 
 \begin{enumerate}

  \item Is there a spectral description of the structure determined by the expectation of the product of the random variables at more than two vertices? In  this case not just the distance matters, the configuration of the vertices has to be fixed. 
  
  \item Is there a spectral description of the moments of the random variables at the vertices? That is, we assign $X_v, X^2_v, \ldots, X^k_v$ to each vertex instead of $X_v$, and we can take covariance matrices of the vectors assign to a pair of vertices. 
  
 \end{enumerate}
 
  \subsection*{Acknowledgement.}  The authors are especially grateful to Russell Lyons for useful comments that led to a shorter proof of Proposition \ref{all:dtv}.  The authors thank P\'eter Csikv\'ari, Viktor Harangi and Mustazee Rahman for reading previous versions of the manuscript and providing valuable suggestions, G\'erard Letac for references,  G\'abor Pete for the question on the covariance structure of $\bar d_2$-limits of factor of i.i.d.\ processes,  and Bal\'azs Szegedy for useful discussions. The first author was partially supported by the National Research, Development and Innovation Office (NKFIH, grant no.  109684), and by the MTA R\'enyi Institute Lend\"ulet Limits of Structures Research Group. The second author was supported by Marie Sk\l odowska-Curie actions "Spectra", by the NSERC Discovery Accelerator Supplements Program and by the MTA R\'enyi ''Lend\"ulet'' Groups and Graphs Research Group.

\bibliographystyle{dcu}
\bibliography{./corrfiid}

\end{document}